\documentclass[11pt]{article}

\usepackage{amsmath}
\usepackage{amsfonts}
\usepackage{amssymb}
\usepackage{amsthm}
\usepackage{graphics}
\usepackage{amscd}
\usepackage{graphicx,epsfig}
\usepackage{color}
\usepackage[all]{xy}
\usepackage{url}

\renewcommand{\epsilon}{\varepsilon}
\newcommand{\R}{\mathbb{R}}

\newcommand{\Z}{\mathbb{Z}}

\renewcommand{\H}{\mathbb{H}}
\renewcommand{\S}{\mathbb{S}}

\newcommand{\T}{\mathbb{T}}

\newcommand{\eps}{\varepsilon}
\newcommand{\der}[2]{\dfrac{\partial #1}{\partial #2}}

\newcommand{\Ome}{\Omega}

\newtheorem{thm}{Theorem}

\renewcommand{\phi}{\varphi}

\newtheorem*{thm*}{Theorem}

\newtheorem{claim}{Claim}
\newtheorem*{claim*}{Claim}

\theoremstyle{remark}

\newtheorem{remarq}{Remark}
\newtheorem*{rem*}{Remark}

\newcounter{remark}

\newcounter{case}

\newcounter{construction}

\newcounter{fact}

\newcounter{step}

\title{On minimal spheres of area $4\pi$ and rigidity}
\author{Laurent Mazet, Harold Rosenberg \thanks{The authors were partially
supported by the ANR-11-IS01-0002 grant.}}
\date{}

\begin{document}
\maketitle
\begin{abstract}
Let $M$ be a complete Riemannian $3$-manifold with sectional curvatures between
$0$ and $1$. A minimal $2$-sphere immersed in $M$ has area at least $4\pi$. If
an embedded minimal sphere has area $4\pi$, then $M$ is isometric to the unit
$3$-sphere or to a quotient of the product of the unit $2$-sphere with $\R$, with the product
metric. We also obtain a rigidity theorem for the existence of hyperbolic cusps.
Let $M$ be a complete Riemannian $3$-manifold with sectional curvatures bounded
above by $-1$. Suppose there is a $2$-torus $T$ embedded in $M$ with mean
curvature one. Then the mean convex component of $M$ bounded by $T$ is a
hyperbolic cusp;,\textit{i.e.}, it is isometric to $T \times \R$ with the
constant curvature $-1$ metric: $e^{-2t}d\sigma_0^2+dt^2$ with $d\sigma_0^2$ a
flat metric on $T$. 
\end{abstract}

\noindent \textit{Keywords: area of minimal sphere, rigidity of $3$-manifolds,
hyperbolic cusp.}

\section{Introduction}
Consider a smooth ($C^\infty$) complete metric on the $2$-sphere $S$ whose
curvature is between $0$ and $1$. It is well known that a simple closed geodesic
in $S$ has length at least $2\pi$ (see \cite{Pog2} or Klingenberg's theorem in
higher dimension \cite{Kli,ChEb}). It is less well known that when such an $S$
has a simple closed geodesic of length exactly $2\pi$, then $S$ is isometric to
the unit $2$-sphere $\S^2_1$. This result is proved in \cite{AnHo}, and the
authors attribute the theorem to E.~Calabi.

With this in mind, we consider what happens in a complete $3$-manifold $M$ with
sectional curvatures between $0$ and $1$ (henceforth we suppose this curvature
condition on $M$, unless stated otherwise). 

Let $\Sigma$ be an embedded minimal $2$-sphere in $M$. Then the Gauss-Bonnet
theorem and the Gauss equation tells us that the area of $S$ is at least $4\pi$:
indeed we have 
\begin{equation}\label{eq:gaussbonnet}
4\pi=\int_\Sigma \bar K_\Sigma=\int \det(A)+K_{T\Sigma}\le \int_\Sigma
1=A(\Sigma)
\end{equation}
with $\det(A)$ the determinant of the shape operator which is non positive. We
prove in Theorem~\ref{th:main1}, that when the area of $\Sigma$ equals $4\pi$,
then $M$ is isometric to the unit $3$-sphere $\S^3_1$ or to a quotient of the product of the
unit $2$-sphere with $\R$, $\S^2_1\times\R$, with the product metric. 

We remark that Theorem~\ref{th:main1} does not hold for embedded minimal tori.
Given $\eps$ greater than zero, there are Berger spheres with curvatures between
$0$ and $1$, which contain embedded minimal tori of area less than $\eps$. But a
minimal sphere always has area at least $4\pi$. 

It would be interesting to know what
happens in higher dimensions. In the unit $n$-sphere $\S^n_1$, a compact minimal
hyper-surface $\Sigma$ always has volume at least the volume of the equatorial
$n-1$ sphere $\S^{n-1}_1$. Is there a rigidity theorem when one allows metrics
on $\S^n$ ($ =M$), of sectional curvatures between $0$ and $1$? Two questions
arise. First, does an embedded minimal hyper-sphere $\Sigma$ in $M$ have volume
at least the volume of $\S^{n-1}_1$. If this is so, and if $\Sigma$ is an
embedded minimal hyper-sphere with volume exactly the volume of $\S^{n-1}_1$, is
$M$ isometric to $\S^n_1$ or to $\S^{n-1}_1\times\R$?

In the same spirit as Theorem~\ref{th:main1}, we prove a rigidity theorem for
hyperbolic cusps. We recall that a $3$ dimensional hyperbolic cusp is a manifold
of the form $T\times\R$ with $T$ a $2$-torus and the hyperbolic metric
$e^{-2t}d\sigma_0^2+dt^2$ with $d\sigma_0^2$ a flat metric on $T$. In
Theorem~\ref{th:main2}, we prove that if $M$ is a complete Riemannian manifold
with sectional curvatures bounded above by $-1$ and $T$ is a constant mean
curvature $1$ torus embedded in $M$ then the mean convex side of $T$ in $M$ is
isometric to a hyperbolic cusp.

\section{Minimal spheres of area $4\pi$ and rigidity of $3$-manifolds}

In this section, we prove a rigidity result for a Riemannian $3$-manifold $M$
whose sectional curvatures are between $0$ and $1$. As explained in the
introduction, any minimal sphere in such a manifold has area at least
$4\pi$.

We denote by $\S^n_1$ the sphere of dimension $n$ with constant sectional
curvature $1$. We then have the following result.

\begin{thm}\label{th:main1}
Let $M$ be a complete Riemannian $3$-manifold whose sectional curvatures satisfy
$0\le K\le 1$. Assume that there exists an embedded minimal sphere $\Sigma$ in
$M$ with area $4\pi$. Then the manifold $M$ is isometric either to the sphere
$\S^3_1$ or to a quotient of $\S^2_1\times\R$.
\end{thm}

\begin{proof}
Let $\Phi$ be the map $\Sigma\times\R\rightarrow M, (p,t)\mapsto
\exp_p(tN(q))$ where $N$ is a unit normal vector field along $\Sigma$. In the
following, we focus on $\Sigma\times\R_+$; by symmetry of the configuration, the
study is similar for $\Sigma\times\R_-$. 

$\Sigma$ is compact, so there is an $\eps$ such that $\Phi$ is an immersion and
even an embedding on $\Sigma\times[0,\eps)$. Let us define  
$$\eps_0=\sup\{\eps>0| \, \Phi\text{ is an immersion on }
\Sigma\times[0,\eps)\};$$
$\eps_0$ can be equal to $+\infty$. Using $\Phi$, we pull back the Riemannian
metric of $M$ to $\Sigma\times[0,\eps_0)$. This metric can be
written $d s^2=d \sigma_t^2+d t^2$ where $d \sigma_t^2$ is a smooth family of
metrics on $\Sigma$. With this metric, $\Phi$ becomes a local isometry from
$\Sigma\times[0,\eps_0)$ to $M$ and $(\Sigma\times[0,\eps_0),ds^2)$ has
sectional curvatures between $0$ and $1$. Moreover, $\Sigma_0$ is minimal and
has area $4\pi$. Actually, we will prove the following facts.
\begin{claim*}
The metric $d \sigma_0^2$ has constant sectional curvature $1$ so $(\Sigma,d
\sigma_0^2)$ is isometric to $\S^2_1$. Moreover, we have two cases
\begin{enumerate}
\item $\eps_0=\pi/2$ and $d \sigma_t^2=\sin^2td \sigma_0^2$ or
\item $\eps_0=+\infty$ and $d \sigma_t^2=d \sigma_0^2$
\end{enumerate}
\end{claim*}

Let us denote by $\Sigma_t=\Sigma\times\{t\}$ the equidistant surfaces. We
denote by $H(p,t)$ the mean curvature of $\Sigma_t$ at the point $(p,t)$ with
respect to the unit normal vector $\partial_t$. We also define $\lambda(p,t)\ge
0$ such that $H+\lambda$ and $H-\lambda$ are the principal curvature of
$\Sigma_t$ at $(p,t)$. We notice that $\lambda=0$ if $\Sigma_t$ is umbilical at
$(p,t)$.  

The surfaces $\Sigma_t$ are spheres so, using the Gauss equation, the
Gauss-Bonnet formula implies:
$$
4\pi=\int_{\Sigma_t}\bar K_{\Sigma_t}=\int_{\Sigma_t} (H+\lambda)(H-\lambda)+
K_t= \int_{\Sigma_t}H^2-\lambda^2+K_t
$$
where $\bar K_{\Sigma_t}$ is the intrinsic curvature of $\Sigma_t$ and $K_t$ is
the sectional curvature of the ambient manifold of the tangent space to
$\Sigma_t$. Since $K_t\le 1$, we obtain the following inequality
\begin{equation}\label{eq:defF}
\int_{\Sigma_t}\lambda^2=\int_{\Sigma_t} H^2+K_t-4\pi\le
\int_{\Sigma_t} H^2+A(\Sigma_t)-4\pi
\end{equation}
where $A(\Sigma_t)$ is the area of $\Sigma_t$. In the following, we denote by
$F(t)$ the right hand side of this inequality.
\begin{claim}\label{cl:vanishing}
$F$ is vanishing on $[0,\eps_0)$.
\end{claim}
Since $\Sigma_0$ is minimal and has area $4\pi$, we have $F(0)=0$. We notice
that this implies that $\lambda(p,0)=0$ so $\Sigma_0$ is umbilical and
$K_{T\Sigma_0}=1$. Thus $(\Sigma_0,d\sigma_0)$ is isometric to $\S^2_1$. 

We have the usual formula:
\begin{equation}\label{eq:riccati}
\der{}{t}A(\Sigma_t)=-\int_{\Sigma_t}2H\quad\text{and}\quad \der{H}{t}=\frac12
(Ric(\partial_t)+|A_t|^2)
\end{equation}
where $A_t$ is the shape operator of $\Sigma_t$ and $Ric$ is the Ricci tensor of
$\Sigma\times[0,\eps_0)$. Since the sectional curvatures of $M\times[0,\eps_0)$
are non-negative, $Ric$ is non-negative. So the second formula above implies
that $H$ is increasing and thus $H\ge 0$ everywhere. Let us now compute and
estimate the derivative of $F$:
\begin{align*}
F'(t)&=\int_{\Sigma_t}(2H\der{H}{t}-2H^3)-\int_{\Sigma_t}2H\\
&=\int_{\Sigma_t}H(Ric(\partial_t)+|A_t|^2-2H^2-2)\\
&=\int_{\Sigma_t}H\big((Ric(\partial_t)-2)+
((H+\lambda)^2+(H-\lambda)^2-2H^2)\big)\\
&=\int_{\Sigma_t}H((Ric(\partial_t)-2)+2\lambda^2)\\
&\le2\int_{\Sigma_t}H\lambda^2
\end{align*}
where the last inequality comes from $Ric(\partial_t)-2\le 0$ because of the
hypothesis on the sectional curvatures. If we choose $\eps<\eps_0$, there is a
constant $C\ge 0$ such that $H\le C$ on $\Sigma\times[0,\eps]$. So for $t\in
[0,\eps]$, using the inequality \eqref{eq:defF}, we get $F'(t)\le 2C F(t)$. Then
$F(t)\le F(0)e^{2Ct}=0$ on $[0,\eps]$. So $F\le 0$ on $[0,\eps_0)$ and, because
of \eqref{eq:defF}, $F=0$ on $[0,\eps_0)$; this finishes the proof of
Claim~\ref{cl:vanishing}.

The first consequence of Claim~\ref{cl:vanishing} is that all the equidistant
surfaces $\Sigma_t$ are umbilical (see inequality \eqref{eq:defF}); so
$\lambda\equiv 0$. In the computation of the derivative of $F$, this implies
that 
$$
\int_{\Sigma_t}H(Ric(\partial_t)-2)=0
$$ 
Since $H(Ric(\partial_t)-2)\le 0$ everywhere, we obtain
\begin{equation}\label{eq:ric=2}
H(Ric(\partial_t)-2)=0 \text{ everywhere.}
\end{equation}
Moreover the umbilicity and \eqref{eq:riccati} implies that $\der{H}{t}=\frac12
Ric(\partial_t)+H^2$. We now prove the following claim
\begin{claim}\label{cl:meancurv}
Let $(p,t)\in\Sigma\times[0,\eps_0)$ ($t>0$) be such that $H(p,t)>0$ then
$H(q,t)>0$ for any $q\in\Sigma$
\end{claim}
In other words, when the mean curvature is positive at a point of an
equidistant, it is positive at any point of this equidistant. We recall that $H$
is increasing in the $t$ variable so when it becomes positive it stays positive.

So assume that $H(p,t)>0$ and consider $\Ome=\{q\in\Sigma|\,H(q,t)>0\}$ which is
a nonempty open subset of $\Sigma$. Let $q\in\Ome$. Since
$H(q,t)>0$, $Ric(\partial_t)(q,t)=2$ by \eqref{eq:ric=2}. Thus
$Ric(\partial_t)(r,t)=2$ for any $r\in\bar\Ome$. So if $r\in\bar\Ome$,
$Ric(\partial_t)(r,s)>0$ for $s< t$, close to $t$ and, by \eqref{eq:riccati},
this implies that $H(r,t)>0$ and $r\in\Ome$. So $\Ome$ is closed and
$\Ome=\Sigma$. This finishes the proof of Claim~\ref{cl:meancurv}.

Let us assume that there is an $\eps_1>0$ such that $H(p,t)=0$ for
$(p,t)\in\Sigma\times[0, \eps_1]$ and $H(p,t)>0$ for any
$(p,t)\in\Sigma\times(\eps_1,\eps_0)$. Because of the evolution equation of $H$,
this implies that $Ric(\partial_t)=0$ on $\Sigma\times[0,\eps_1]$. On
$\Sigma\times(\eps_1,\eps_0)$, we have $Ric(\partial_t)=2$ because of
\eqref{eq:ric=2}. So by continuity of $Ric(\partial_t)$, we get a contradiction
and then we have two possibilities
\begin{enumerate}
\item $H=0$ on $\Sigma\times[0,\eps_0)$ and $Ric(\partial_t)=0$ on
$\Sigma\times[0,\eps_0)$.
\item $H>0$ on $\Sigma\times(0,\eps_0)$ and $Ric(\partial_t)=2$ on
$\Sigma\times[0,\eps_0)$.
\end{enumerate}

In the first case, this implies that the sectional curvature of any $2$-plane
orthogonal to $\Sigma_t$ is zero. Thus $d\sigma_t^2=d\sigma_0^2$. Since the
map $\Phi$ ceases to be an immersion only if $d\sigma_t^2$ becomes singular this
implies that $\eps_0=+\infty$. Thus $\Sigma\times\R_+$ with the induced metric
is isometric to $\S^2_1\times\R_+$ and $\Phi$ is a local isometry from
$\S^2_1\times\R_+$ to $M$.

In the second case, the sectional curvature of any $2$-plane orthogonal to
$\Sigma_t$ is equal to $1$. Thus $d\sigma_t^2=\sin^2t d\sigma_0$ and
$\eps_0=\pi/2$. This also implies that $\Phi(p,\pi/2)$ is a point. So
$\Sigma\times[0,\pi/2]$ with the metric $ds^2$ is isometric to a hemisphere of
$\S^3_1$ and the map $\Phi$ is a local isometry from that hemisphere to $M$.

Doing the same study for $\Sigma\times\R_-$, we get in the first case a local
isometry $\Phi:\S^2_1\times \R\rightarrow M$ and in the second case a local
isometry $\Phi:\S^3_1\rightarrow M$. Since $\S^2_1\times\R$ and $\S^3_1$ are
simply connected, $\Phi$ is then the universal cover of $M$ and $M$ is then isometric to a quotient  of $\S^2_1\times\R$ or $\S^3_1$. Since $\Phi$ is
injective on $\Sigma$ this implies that in the second case, $\Phi$ is actually
injective and then a global isometry.
\end{proof}

\begin{remarq}
In the proof, since $\Phi$ is injective on $\Sigma$, the possible quotients of $\S^2_1\times\R$ are either $\S^2_1\times\R$ or its quotient by the subgroup generated by an isometry of the form $\S^2_1\times\R\rightarrow\S^2_1\times\R; (p,t)\mapsto(\alpha(p),t+t_0)$ with $\alpha$ an isometry of $\S^2_1$ and $t_0\neq 0$. 
\end{remarq}

\begin{remarq}
Something can be said about constant mean curvature $H_0$ spheres
in a Riemannian $3$-manifold with sectional curvatures between $0$ and $1$.
Indeed, the computation \eqref{eq:gaussbonnet} implies
that the area of $\Sigma$ is larger than $\frac{4\pi}{1+H_0^2}$, which is the
area of a geodesic sphere in $\S^3_1$ of mean curvature $H_0$. Moreover, if
$\Sigma$ has area $\frac{4\pi}{1+H^2}$, the above proof can be adapted to prove
that the mean convex side of $\Sigma$ is isometric to a spherical cap of
$\S^3_1$ with constant mean curvature $H_0$ (see Theorem~\ref{th:main2} below,
for a similar result in the hyperbolic case). 
\end{remarq}

\begin{remarq}
Let $M$ be a Riemannian $n$-manifold whose sectional curvatures are between $0$
and $1$ and let $\Sigma$ be a minimal $2$-sphere in $M$. A computation similar
to \eqref{eq:gaussbonnet} proves also that the area of $\Sigma$ is larger than
$4\pi$. It also implies that, if $\Sigma$ has area $4\pi$, $\Sigma$ is
totally geodesic and isometric to $\S^2_1$.
\end{remarq}

\section{Existence of hyperbolic cusps}

Let $(\T^2,g)$ be a flat $2$ torus, the manifold $\T^2\times\R_+$ with the
complete Riemannian metric $e^{-2t}g+dt^2$ is a hyperbolic $3$-dimensional
cusp. $\T^2\times\R$ is actually isometric to the quotient of a horoball of
$\H^3$ by a $\Z^2$ subgroup of isometries of $\H^2$ leaving the horoball
invariant. Any $\T^2\times\{t\}$ has constant mean curvature $1$. The
following theorem says that, in certain $3$-manifolds, a constant mean curvature
$1$ torus is necessarily the boundary of a hyperbolic cusp.

\begin{thm}\label{th:main2}
Let $M$ be a complete Riemannian $3$-manifold with its sectional curvatures
satisfying $K\le -1$. Assume that there exists a constant mean curvature $1$
torus $T$ embedded in $M$. Then $T$ separates $M$ and its mean convex side is
isometric to a hyperbolic cusp. 
\end{thm}

As a consequence, the existence of this torus implies that $M$ can not be
compact. The proof uses the same ideas as in Theorem~\ref{th:main1}

\begin{proof}
Let us consider the map $\Phi:T\times\R_+\rightarrow
M,(p,t)\mapsto\exp_p(tN(p))$ where $N$ is the unit normal vector field
normal to $T$ such that $N$ is the mean curvature vector of $T$. Let us define 
$$\eps_0=\sup\{\eps>0| \, \Phi\text{ is an immersion on } T\times[0,\eps)\}.$$
Using $\Phi$, we pull back the Riemannian metric of $M$ to $T\times[0,\eps_0)$;
it can be written $ds^2=dt^2+d\sigma_t^2$. We define $T_t=T\times\{t\}$ the
equidistant surfaces to $T_0$. We also denote by $H(p,t)$ the mean curvature of
the equidistant surfaces at $(p,t)$ with respect to $\partial_t$. We finally
define $\lambda(p,t)$ such that $H+\lambda$ and $H-\lambda$ are the principal
curvatures of $T_t$ at $(p,t)$.

The surfaces $T_t$ are tori so, by the Gauss equation and the Gauss-Bonnet
formula, we have
$$
0=\int_{T_t}\bar K_{T_t}=\int_{T_t}H^2-\lambda^2+K_t
$$ 
where $K_t$ is the sectional curvature of the ambient manifold of the
tangent space to $T_t$. Since $K_t\le -1$, we obtain the inequality
$$
\int_{T_t}\lambda^2=\int_{T_t}H^2+K_t\le\int_{T_t}H^2-A(T_t)
$$
Let $F(t)$ denote the right hand term of the above inequality. By hypothesis,
$H(p,0)=1$ so $F(0)=0$ and $F(t)\ge0$ for any $t\ge 0$. Let us compute the
derivative of $F$
\begin{align*}
F'(t)&=\int_{T_t}(2H\der{H}{t}-2H^3)+\int_{T_t}2H\\
&=\int_{T_t}H(Ric(\partial_t)+|A_t|^2-2H^2+2)\\
&=\int_{T_t}H((Ric(\partial_t)+2)+2\lambda^2)
\end{align*}

Since $H(p,0)=1$, we can consider $\eps\in(0,\eps_0)$ such that $0< H\le C$ on
$T\times[0,\eps]$. Since $Ric(\partial_t)+2\le 0$ we get:
$$
F'(t)\le\int_{T_t}2H\lambda^2\le 2CF(t)
$$
Thus $F(t)\le F(0)e^{2Ct}$ for $t\in[0,\eps]$; this implies $F(t)=0$ on that
segment. We then obtain $\lambda=0$ on $T\times[0,\eps]$ (the equidistant
surfaces are umbilical) and $Ric(\partial_t)=-2$ since $H>0$. Thus $H$ satisfies
the differential equation $\der{H}{t}=-2+2H^2$. This gives that $H=1$ on
$T\times[0,\eps]$ since $H=1$ on $T_0$. Thus we can let $\eps$ tend to $\eps_0$
to obtain that $F(t)=0$ on $[0,\eps_0)$ and $Ric(\partial_t)=-2$ and $H=1$ on
$T\times[0,\eps_0)$. Since $0=\int_{T_t}H^2+K_t$ and $K_t\le -1$, it follows
that $K_t=-1$ for all $t$ in the interval. We then have proved that
the sectional curvature of $T\times[0,\eps_0)$ with the metric $ds^2$ is equal
to $-1$ for any $2$-plane. Moreover, we get that $d\sigma_0^2$ is flat and that
$d\sigma_t^2=e^{-2t}d\sigma_0^2$. This implies that $\Phi$ is actually an
immersion on $T\times\R_+$ ($\eps_0=+\infty$) and $T\times\R_+$ is isometric to
a hyperbolic cusp. $\Phi$ is then a local isometry from this hyperbolic cusp to
$M$. 

To finish the proof, let us prove that $\Phi$ is in fact injective. If this is not
the case, let $\eps_1>0$ be the smallest $\eps$ such that $\Phi$ is not
injective on $T\times[0,\eps]$. This implies that there exist $p$ and $q$ in $T$
such that 
\begin{itemize}
\item either $\Phi(p,0)=\Phi(q,\eps_1)$
\item or $\Phi(p,\eps_1)=\Phi(q,\eps_1)$ (with $p\neq q$ in this case).
\end{itemize}
Let $U$ and $V$ be respective neighborhoods of $(p,0)$ (or $(p,\eps_1)$) in
$T_0$ (or $T_{\eps_1}$) and $(q,\eps_1)$ in $T_{\eps_1} $ such that $\Phi$ is
injective on them. Since $\eps_1$ is the smallest one, $\Phi(U)$ and $\Phi(V)$
are two constant mean curvature $1$ surfaces in $M$ that are tangent at
$\Phi(q,\eps_1)$. Moreover, in the first case, $\Phi(U)$ is included in the mean
convex side of $\Phi(V)$ so by the maximum principle $\Phi(U)=\Phi(V)$. Thus
$\Phi(T_0)$ would be equal to $\Phi(T_{\eps_1})$ which is impossible since these
two surfaces do not have the same area. In the second case, $\Phi(U)$ is
included in the mean convex side of $\Phi(V)$ and then $\Phi$ is not injective
on $T_s$ for $s$ near $t$ $s<t$, which is a contradiction.
\end{proof}

\noindent \textsc{Laurent Mazet}\\
Universit\'e Paris-Est \\
LAMA (UMR 8050), UPEC, UPMLV, CNRS\\
F-94010, Cr\'eteil, France

\noindent \verb?laurent.mazet@math.cnrs.fr?

\bigskip

\noindent \textsc{Harold Rosenberg}\\
IMPA\\
Estrada Dona Castorina 110\\
Rio de Janeiro / Brasil 22460-320

\noindent \verb?hrosen@free.fr?

\end{document}